\documentclass{article}
\usepackage{amsfonts,latexsym}
\usepackage{amsmath, times}
\usepackage{amssymb}
\usepackage{amsthm}

\newcommand{\R}{\mathbb R}

\newcommand{\del}{\partial}

\newtheorem{theorem}{Theorem}[section]
\newtheorem{lemma}{Lemma}
\newtheorem{proposition}{Proposition}[section]

\newtheorem{definition}{Definition}[section]

\newlength{\defbaselineskip}
\newcommand{\setlinespacing}[2]%
          {\setlength{\baselineskip}{#1 \defbaselineskip}}

\makeatletter

\@addtoreset{equation}{section} \makeatother 
\begin{document}
\begin{center}
 {\Large   {Fourth order weighted elliptic problem under exponential
nonlinear growth }}
\end{center}
\vspace{0.2cm}

\begin{center}
 Brahim Dridi $^{(1)}$ and  Rached Jaidane $^{(2)}$

 \
\noindent\footnotesize$^{(1)}$
 Department of Mathematics, El Manar Preparatory Institute for Engineering Studies Tunis, University of Tunis El Manar, Tunisia.
Address e-mail: dridibr@gmail.com\\
\noindent\footnotesize $^{(2)}$ Department of Mathematics, Faculty of Science of Tunis, University of Tunis El Manar, Tunisia.\\
 Address e-mail: rachedjaidane@gmail.com\\
\end{center}

\vspace{0.5cm}
\begin{abstract}
We deal with  nonlinear   weighted biharmonic problem
in the unit ball of $\mathbb{R}^{4}$. The weight is
of logarithm type.
 The nonlinearity is critical  in view of Adam's
 inequalities in the weighted Sobolev space $W^{2,2}_{0}(B,w)$. We  prove the existence of non trivial
 solutions via the critical point theory. The main difficulty is
the loss of compactness due to the critical exponential growth of the nonlinear
term $f$. We give a new growth condition and we point out its importance
 for checking the Palais-Smale compactness condition.
\end{abstract}

\noindent {\footnotesize\emph{Keywords:} Adam's inequality,  Moser-Trudinger's inequality, Nonlinearity of exponential growth, Mountain pass method, Compactness level.\\
\noindent {\bf $2010$ Mathematics Subject classification}: $35$J$20$, $35$J$30$, $35$K$57$, $35$J$60$.}

%1%%%%%%%%%%%%%%%%%%%%%%%%%%%%%%%%%%%%%%%%%%%%%%%%%%%%%%%%%%%%%%%%%%%%%%%%%%%%%%%%%%%%%%%%%%%1111111
\section{Introduction and Main results}
In  this paper, we consider  the  following elliptic nonlinear problem:
\begin{equation}\label{eq:1.1.1}
  \displaystyle \left\{
      \begin{array}{rclll}
 L(u):=\Delta ( w(x)  \Delta u)&= &f(x,u) &\mbox{ in }& B, \\
        u=\frac{\partial u}{\partial n}  & =&0 & \mbox{ on }& \partial B,
      \end{array}
          \right.
\end{equation}
where $B=B(0,1)$ is the unit open ball in $\R^{4}$. The weight is given by
\begin{equation}\label{eq:1.2}
w(x)=\big(\log \frac{e}{|x|}\big)^{\beta}, \beta\in (0,1) ,
\end{equation}
The nonlinearity $f(x, t)$ is continuous in $B\times \mathbb{R}$ and behaves like
$\exp\{\alpha t^{\frac{2}{1-\beta}}\}$ as $ t\rightarrow+ \infty$, for some $\alpha >0$ and where  $\frac{\partial u}{\partial n}$ denotes the outer normal derivative of $u$
on $\partial B$.\\
Problems of critical exponential growth in second order elliptic equations in dimension $N=2$
\begin{equation*}-\Delta u=f(x,u)~~\mbox{in}~~\Omega\subset \mathbb{R}^{2}.\end{equation*} have been studied considerably  \cite{Adi1,FMR,LL1,MS}. In dimension $N\geq 2$, the critical exponential growth is given by the well known
Trudinger-Moser inequality \cite{JMo,NST} $$\displaystyle\sup_{\int_{\Omega} |\nabla u|^{N}\leq1}\int_{\Omega}e^{\alpha|u|^{\frac{N}{N-1}}}dx<+\infty~~\mbox{if and only if}~~\alpha\leq \alpha_{N},$$
where $\alpha_{N}=\omega_{N-1}^{\frac{1}{N-1}}$  with $\omega_{N-1}$ is the area of the unit sphere $S^{N-1}$ in $\mathbb{R}^{N}$.\\
Later, the Trudinger-Moser inequality was improved to weighted inequalities \cite{ CR1, CR2}. The influence of the weight in the Sobolev norm was studied as the compact embedding \cite{Kuf}.\\
  When the weight is of logarithmic type, Calanchi and Ruf \cite{CR3} extend the Trudinger-Moser inequality and proved the following results in the weighted Sobolev  space for radial functions $$W_{0,rad}^{1,N}(B,\rho)=cl\{u \in
C_{0,rad}^{\infty}(B)~~|~~\int_{B}|\nabla u|^{N}\rho(x)dx <\infty\}:$$
\begin{theorem}\cite{CR2} \label{th1.1}\begin{itemize}\item[$(i)$] ~~Let $\beta\in[0,1)$ and let $\rho$ given by $ \rho(x)=\big(\log \frac{1}{|x|}\big)^{\beta}$, then
  \begin{equation*}
 \int_{B} e^{|u|^{\gamma}} dx <+\infty, ~~\forall~~u\in W_{0,rad}^{1,N}(B,\rho),~~
  \mbox{if and only if}~~\gamma\leq \gamma_{N,\beta}=\frac{N}{(N-1)(1-\beta)}=\frac{N'}{1-\beta}
 \end{equation*}
and
 \begin{equation*}
 \sup_{\substack{u\in W_{0,rad}^{1,N}(B,\rho) \\ \int_{B}|\nabla u|^{N}w(x)dx\leq 1}}
 \int_{B}~e^{\alpha|u|^{\gamma_{N,\beta} }}dx < +\infty~~~~\Leftrightarrow~~~~ \alpha\leq \alpha_{N,\beta}=N[\omega^{\frac{1}{N-1}}_{N-1}(1-\beta)]^{\frac{1}{1-\beta}}
 \end{equation*}
where $\omega_{N-1}$ is the area of the unit sphere $S^{N-1}$ in $\R^{N}$ and $N'$ is the H$\ddot{o}$lder conjugate of $N$.
\item [$(ii)$] Let $\rho$ given by $\rho(x)=\big(\log \frac{e}{|x|}\big)^{N-1}$, then
  \begin{equation*}\label{eq:71.5}
 \int_{B}exp\{e^{|u|^{\frac{N}{N-1}}}\}dx <+\infty, ~~~~\forall~~u\in W_{0,rad}^{1,N}(B,\rho)
 \end{equation*} and
 \begin{equation*}\label{eq:71.6}
\sup_{\substack{u\in W_{0,rad}^{1,N}(B,\rho) \\  \|u\|_{\rho}\leq 1}}
 \int_{B}exp\{\beta e^{\omega_{N-1}^{\frac{1}{N-1}}|u|^{\frac{N}{N-1}}}\}dx < +\infty~~~~\Leftrightarrow~~~~ \beta\leq N,
 \end{equation*}
where $\omega_{N-1}$ is the area of the unit sphere $S^{N-1}$ in $\R^{N}$ and $N'$ is the H$\ddot{o}$lder conjugate of $N$.\end{itemize}
\end{theorem}
 These results opened the way to study second order weighted elliptic problems in dimension $N\geq2$ . We cite the work of Calanchi et all  \cite{CRS}, $N=2$

$$
 \displaystyle \left\{
      \begin{array}{rclll}
   -\nabla. (\nu(x)\nabla u ) &=& \ f(x,u)& \mbox{in} & B \\
       u &>&0 &\mbox{in }& B\\
        u&=&0 &\mbox{on }&  \partial B,
      \end{array}
         \right.
      $$
with the weight $\nu(x)=\log(\frac{e}{|x|})$ and where the function $f(x, t)$ is continuous in $B \times \mathbb{R}$ and behaves like $\exp\{e^{\alpha{t^{2}}}\}~~\mbox{as}~~t\rightarrow+\infty$, for some $\alpha>0$.\\ Also, recently, Deng et all \cite{DHT} and Zhang\cite{Z}  studied the following problem \begin{equation*}
\displaystyle \left\{
\begin{array}{rclll}
-\textmd{div} (\rho(x)|\nabla u|^{N-2}  \nabla u) &=&  \displaystyle f(x,u)& \mbox{in} & B \\

u&=&0 &\mbox{on }&  \partial B,
\end{array}
\right.
\end{equation*} where $ N\geq 2$, the function $f(x, t)$ is continuous in $B \times \mathbb{R}$ and behaves like $\exp\{e^{\alpha{t^{\frac{N}{N-1}}}}\}~~\mbox{as}~~t\rightarrow+\infty$, for some $\alpha>0$. The authors proved that there is a non-trivial solution to this problem using Mountain Pass theorem.\\ Also, we mention that Baraket et all \cite{BJ} studied the following non-autonomous weighted elliptic equations \begin{equation*}\label{eq:1.1}
\displaystyle \left\{
\begin{array}{rclll}
L:=-\textmd{div} (\rho(x)|\nabla u|^{N-2}  \nabla u)+\xi(x)|u|^{N-2}u &=&  \displaystyle f(x,u)& \mbox{in} & B \\
u &>&0 &\mbox{in }& B\\
u&=&0 &\mbox{on }&  \partial B,
\end{array}
\right.
\end{equation*}
where B is the unit ball of $\mathbb{R}^{N}$, $ N>2$ , $f(x, t)$ is continuous in $B\times \mathbb{R}$ and behaves like
$\exp\{e^{\alpha t^{\frac{N}{N-1}}}\}$ as $ t\rightarrow+ \infty$, for some $\alpha >0$. $\xi :B\rightarrow \mathbb{R}$ is a positive continuous function satisfying some conditions. The weight $\rho(x)$ is given by
$\rho(x)=\big(\log \frac{e}{|x|}\big)^{N-1}\cdot$\\
  The biharmonic equation in dimension $N>4$ \begin{equation*}\Delta^{2} u=f(x,u)~~\mbox{in}~~\Omega\subset \mathbb{R}^{N},\end{equation*}where the nonlinearity $f$ has subcritical and critical polynomial growth of power less than $\frac{N+4}{N-4}$, have been extensively studied \cite{BGT,EFJ,GGS,RW} .\\  \\
 For bounded domains $\Omega\subset \mathbb{R}^{4}$, in \cite{Ada,BRFS} the authors proved the following Adams' inequality  \begin{equation*}
 \sup_{\substack{u\in S }}
 \int_{\Omega}~(e^{\alpha u^{2})}-1)dx < +\infty~~~~\Leftrightarrow~~~~ \alpha\leq 32 \pi^{2}
 \end{equation*}
 where $$S=\{u\in W^{2,2}_{0}(\Omega)~~|~~\displaystyle\big(\int_{\Omega}|\triangle u|^{2}dx \big)^{\frac{1}{2}}\leq 1\}.$$
This last result opened the way to study fourth-order problems with subcritical or critical nonlinearity (see \cite{FS} , \cite{Chen}).\\
 We study the existence of the nontrivial solutions when the nonlinear terms have
the critical exponential growth in the sense of Adams inequalities \cite{WZ}. Our approach
is variational methods such as the Mountain Pass Theorem with Palais-Smale
condition combining with a concentration compactness result. \\ More precisely, let $\Omega \subset \R^{4}$ be a bounded domain and $w\in L^{1}(\Omega)$ be a nonnegative function, the weighted sobolev space is defined as
$ W_{0}^{2,2}(\Omega,w)=cl\{u\in
C_{0}^{\infty}(\Omega)~~|~~\displaystyle\int_{\Omega}|\triangle u|^{2}w(x)dx <\infty\}.$
We will restrict our attention to radial functions and then consider the subspace
 \begin{equation}\label{eq:1.3}
\mathbf{W}=W_{0,rad}^{2,2}(B,w)=cl\{u \in
C_{0,rad}^{\infty}(B)~~|~~\int_{B}|\triangle u|^{2}w(x)~~dx <\infty\},
 \end{equation}
 equipped with norm
  $$\|u\|=\displaystyle\big(\int_{B}|\triangle u|^{2}w(x)dx\big)^{\frac{1}{2}},~~w(x)=\big(\log \frac{e}{|x|}\big)^{\beta}$$ which comes from the scalar product
$$<u,v>=\int_{B}\Delta u.\Delta v ~(\log\frac{e}{\vert x\vert})^{\beta}~dx.$$
   The norm $\|u\|=\displaystyle\big(\int_{B}|\triangle u|^{2}w(x)dx\big)^{\frac{1}{2}},$ and   $$\|u\|_{W_{0,rad}^{2,2}(B,w)}=\displaystyle\big(\int_{B}u^{2}dx +\int_{B} |\nabla u|^{2}~dx+\int_{B}|\triangle u|^{2}w(x)dx \big)^{\frac{1}{2}}$$ are equivalent (see Lemma \ref{lem31}).
\\The choice of the weight  and the space $W_{0,rad}^{2,2}(B,w)$ are  motivated by the following inequality of Adam's type.
 %%%%%%%%%%%%%%%%%%%%%%%%%%%%%%%%%%%%%%%%%%%%%%%%%%%%%%%%%%%%%%%%%%%%%%%%%%%%%%ùù
\begin{theorem}\cite{WZ} \label{th1.1} ~~Let $\beta\in(0,1)$ and let $w$ given by (\ref{eq:1.2}), then

 \begin{equation}\label{eq:1.4}
 \sup_{\substack{u\in W_{0,rad}^{2,2}(B,w) \\  \|u\|\leq 1}}
 \int_{B}~e^{\displaystyle\alpha|u|^{\frac{2}{1-\beta} }}dx < +\infty~~~~\Leftrightarrow~~~~ \alpha\leq \alpha_{\beta}=4[8\pi^{2}(1-\beta)]^{\frac{1}{1-\beta}}
 \end{equation}

 \end{theorem}

Let $\gamma:=\displaystyle\frac{2}{1-\beta}$. In view  of inequality (\ref{eq:1.4}), we say that $f$ has critical growth at $+\infty$ if there exists some $\alpha_{0}>0$,
\begin{equation}\label{eq:1.5}
\lim_{s\rightarrow +\infty}\frac{|f(x,s)|}{e^{\alpha s^{\gamma}}}=0,~~~\forall~\alpha~~\mbox{such that}~~ \alpha>\alpha_{0} ~~~~
\mbox{and}~~~~\lim_{s\rightarrow +\infty}\frac{|f(x,s)|}{e^{\alpha s^{\gamma}}}=+\infty,~~\forall~ \alpha<\alpha_{0} .
\end{equation}
Let us now state our results.
We suppose that $f(x,t)$ satisfies the following hypothesis:
%%%%%%%%%%%%%%%%%%%%%%%%%%%%%%%%%%%%%%%%%%%%%%%%%%%ù
\begin{enumerate}
\item[$(H_{1})$] $f: B \times \mathbb{R}\rightarrow\mathbb{R}$ is continuous, positive, radial in $x$ and $f(x,t) = 0$ for $t \leq 0;$
\item[$(H_{2})$] There exists $t_{0} > 0$ and $M > 0$ such that for all
$t>t_{0}$ and for all $x\in B$ we have
 $$0 < F(x,t)\leq Mf(x,t),$$
  where
$$F(x,t)=\displaystyle\int_{0}^{t}f(x,s)ds;$$
\item [$(H_{3})$] $0<F(x,t)\leq \displaystyle\frac{1}{2}f(x,t)t$,~~ $\forall
t>0, \forall x\in B.$
\item[$(H_{4})$] $~~~\displaystyle\limsup_{t\rightarrow 0}\frac{2 F(x,t)}{t^{2}}< \lambda_{1} ~~~~ \mbox{uniformly in}~~ x,$
\end{enumerate}
We denote by
$$\lambda_{1}=\inf_{
\substack{u\in \mathbf{W} \\  u\neq0} }\displaystyle\frac{\displaystyle\int_{B}|\triangle u|^{2}w(x)dx}{\displaystyle\int_{B}|u|^{2}dx},$$
the first eigenvalue of $( L, \mathbf{W})$. It is well known that $\lambda_{1}$ is isolated simple positive eigenvalue and has a positive bounded associated eigenfunction, \cite{DKN}.\\
%%%%%%%%%%%%%%%%%%%%%%%%%%%%%%%%%%%%%%%%%%%%%%%%%%%%%%%%%%%%%%%%%%%%%%%%%ù
We say that u is a solution to the problem (\ref{eq:1.1.1}), if u is a weak solution in the following sense:
\begin{definition}\label {def2.1}
We say that a function $u\in \mathbf{W}$ is a solution of the problem (\ref{eq:1.1.1}) if
\begin{equation*}\label {eq:2.2}
\int_{B}\triangle u .\triangle\varphi~ w(x) dx =
\int_{B}f(x,u) \varphi dx,~~\forall~\varphi \in \mathbf{W}.
\end{equation*}
\end{definition}
Let $\mathcal{J}:\mathbf{W} \rightarrow \R$ be the functional given by
 \begin{equation}\label{eq:1.6}
\mathcal{J}(u)=\frac{1}{2}\int_{B}|\triangle u|^{2}w(x)dx-\int_{B}F(x,u)dx,
\end{equation}
where
$$ F(x,t)=\displaystyle\int_{0}^{t}f(x,s)ds.$$
It is well-known that seeking a
weak solution of (\ref{eq:1.1.1}) is equivalent to finding a nonzero critical point of $\mathcal{J}$.\\
%%%%%%%%%%%%%%%%%%%%%%%%%%%%%%%%%%%%%%%%%%%%%%%%%%%%%%%%%%%%ùùù
 Our Euler–Lagrange functional does not
satisfy the Palais–Smale condition at all level anymore. To overcome the verification
of compactness of Euler–Lagrange functional at some suitable level,
we construct Adams type functions,
which are extremal to the
inequality (\ref{eq:1.4}). Our result is
as follows :
\begin{theorem}\label{th1.3}~~
Assume that $f(x,t)$ has a  critical growth at $+\infty$ for some $\alpha_{0}$ and  satisfies
the conditions $(H_{1}), (H_{2}), (H_{3})$ and $(H_{4})$. If in addition $f(x,t)$ satisfies
the asymptotic condition  $$(H_{5})~~~~~~\displaystyle\lim_{t\rightarrow \infty}\frac{f(x,t)t}{e^{\alpha_{0}t^{\gamma}}}\geq
\gamma_{0}~~~~\mbox{ uniformly in}~~ x, ~~\mbox{with}~~~~\gamma_{0}>  \frac{1024(1-\beta)}{\alpha^{1-\beta}_{0}},$$ then the
problem (\ref{eq:1.1.1}) has a nontrivial solution.
\end{theorem}

  In general the study of  fourth order partial differential equations is considered an interesting topic. The interest in studying such equations was stimulated by their applications in micro-electro-mechanical systems, phase field models of multi-phase systems, thin film theory,surface diffusion on solids, interface dynamics, flow in Hele-Shaw cells, see \cite{D, FW, M}. However  many applications are generated
by the weighted elliptic problems, such as  the study of traveling waves in suspension bridges, radar imaging  (see, for example \cite{AEG, LL}).\\
The main reason for this study is that, to our knowledge, there are few research taking into account both
this type of non-linearity  for a non-linear fourth order elliptic equation in the framework of
Sobolev weighted spaces.\\
\\This paper is organized as follows:\newline In Section 2,  we present some necessary preliminary knowledge
  about working space, and we
prove that the energy $\mathcal{J}$ satisfied the two geometric properties.  Section 3 is devoted for the  compactness analysis. More precisely, we prove a concentration compactness result of  Lions type and identify the first compactness level of the energy $\mathcal{J}$. Finally,  we fulfil the proof of the main results in section 4.\newline
In this work, the constant C may change from line to another and sometimes we index the constants in
order to show how they change.
%2%%%%%%%%%%%%%%%%%%%%%%%%%%%%%%%%%%%%%%%%%%%%%%%%%%%%%%%%%%%%%%%%%%%%%%%%%%%%%%%%%%%%%%%%%%%%%%%%%%%%%%%2%%%%%%%%%%%%%%%%%%%%%%%%%%%%%%%%%%%%%%%%
\section{Functional setting and Variational formulation }
\subsection{Functional setting}
Let $\Omega \subset \R^{N}$, $N\geq2$,  be a bounded domain in $\R^{N}$ and let $w\in L^{1}(\Omega)$ be a nonnegative function. To deal with weighted operator, we need to introduce some functional spaces $L^{p}(\Omega,w)$, $W^{m,p}(\Omega,w)$, $W_{0}^{m,p}(\Omega,w)$ and some of their properties that will be used later. Let $S(\Omega)$ be the set of all measurable real-valued functions defined on $\Omega$ and two measurable functions are considered as the same element if they are equal almost everywhere.\\\\
Following  Drabek et al. and Kufner in \cite{DKN}, the weighted Lebesgue space $L^{p}(\Omega,w)$ is defined as follows:
$$L^{p}(\Omega,w)=\{u:\Omega\rightarrow \R ~\mbox{measurable};~~\int_{\Omega}w(x)|u|^{p}~dx<\infty\}$$
for any real number $1\leq p<\infty$.\\
This is a normed vector space equipped with the norm
$$\|u\|_{p,w}=\Big(\int_{\Omega}w(x)|u|^{p}~dx\Big)^{\frac{1}{p}}.$$
For $m\geq 2$, let $w$ be a given family of weight functions $w_{\tau}, ~~|\tau|\leq m,$ $w=\{w_{\tau}(x)~~x\in\Omega,~~|\tau|\leq m\}.$\\

In \cite{DKN}, the  corresponding weighted Sobolev space was  defined as
$$ W^{m,p}(\Omega,w)=\{ u \in L^{p}(\Omega), D^{\tau} u \in L^{p} (\Omega)~~  \forall ~~1\leq|\tau|\leq m-1 , D^{\tau} u   \in L^{p}(\Omega,w) ~~  \forall ~~|\tau|=m  \}$$
endowed with the following norm:

\begin{equation*}\label{eq:2.2}
\|u\|_{W^{m,p}(\Omega,w)}=\bigg(\sum_{ |\tau|\leq m-1}\int_{\Omega}|D^{\tau}u|^{p}dx+\displaystyle \sum_{ |\tau|= m}\int_{\Omega}|D^{\tau}u|^{p}\omega(x) dx\bigg)^{\frac{1}{p}},
\end{equation*}
where $w_{\tau}=1~~\mbox{for all}~~|\tau|< k,$ $w_{\tau}=\omega~~\mbox{for all}~~|\tau|=k$.\\

If we suppose also that $w(x)\in L^{1}_{loc}(\Omega)$, then $C^{\infty}_{0}(\Omega)$ is a subset of $W^{m,p}(\Omega,w)$ and we can introduce the space $$W^{m,p}_{0}(\Omega,w)$$
as the closure of $C^{\infty}_{0}(\Omega)$ in $W^{m,p}(\Omega,w).$\\
$(L^{p}(\Omega,w),\|\cdot\|_{p,w})$ and $(W^{m,p}(\Omega,w),\|\cdot\|_{W^{m,p}(\Omega,w)})$ are separable, reflexive Banach spaces provided that $w(x)^{\frac{-1}{p-1}} \in L^{1}_{loc}(\Omega)$.\\
 For $w(x)=1$, one finds the standard Sobolev spaces $W^{m,p}(\Omega)$, $W_{0}^{m,p}(\Omega)$ and the Lebesgue spaces  $L^{p}(\Omega)$.\\
Our space setting is $$\mathbf{W}=\{u\in W_{0,rad}^{2,2}(B,w)
~~|~~\displaystyle\int_{\Omega}|\triangle u|^{2}w(x)dx <\infty\}.$$    $\mathbf{W}$  is   equipped with norm
  $$\|u\|=\displaystyle\big(\int_{B}|\triangle u|^{2}w(x)dx\big)^{\frac{1}{2}},$$ which comes from the scalar product
$$<u,v>=\int_{B}\Delta u.\Delta v ~(\log\frac{e}{\vert x\vert})^{\beta}~dx.$$
We have the following result:
\begin{lemma}\label{lem31}~$(\mathbf{W}, \|.\|_{W_{0,rad}^{2,2}(B,w)})$ is a Banach space and the norm $\| .\|$  is equivalent in $\mathbf{W}$ to the norm $\|.\|_{W_{0,rad}^{2,2}(B,w)}$.
\end{lemma}
\begin{proof}The Sobolev weighted space $(\mathbf{W},\|.\|_{W_{0,rad}^{2,2}(B,w)})$ is a normed linear space. In order to prove that it is a Banach space, let $\{u_{n}\}$ be a Cauchy sequence  that
 $$\|u_{n}-u_{m}\|_{\mathbf{W}_{0,rad}^{2,2}(B,w)}\rightarrow 0~~\mbox{as}~~n,m\rightarrow +\infty.$$
 Therefore
$\{u_{n}\}$ is also a Cauchy sequence in $(W_{0,rad}^{2,2}(B,w),\|.\|_{W_{0,rad}^{2,2}(B,w)})$.\\
 By the completeness of the last space, there exists $u\in W_{0,rad}^{2,2}(B,w)$ such that
 \begin{equation}\label{eq:2.1}
 \|u_{n}-u\|_{W_{0,rad}^{2,2}(B,w)}\rightarrow0~~\mbox{as}~~n\rightarrow +\infty.
 \end{equation}
Since  $\|u\|^{2}_{W_{0,rad}^{1,2}(B)}= \displaystyle\int_{B}|\nabla u|^{2}dx$, then $$\|u\|_{W_{0,rad}^{1,2}(B)}\leq \|u\|_{W_{0,rad}^{2,2}(B,w)},$$
 for all $u\in \mathbf{W}$, the sequence $\{u_{n}\}$ is also a Cauchy sequence in $(W^{1,2}_{0,rad}(B),\|.\|_{W^{1}(B)})$. By the completeness of $((W^{1,2}_{0,rad}(B),\|.\|_{W^{1,2}_{0,rad}(B)}$ there exists $v\in W^{1,2}_{0,rad}(B)$ such that
\begin{equation}\label{eq:2.2}\|u_{n}-v\|_{W^{1,2}_{0,rad}(B)}\rightarrow0~~\mbox{as}~~n\rightarrow+\infty
 \end{equation}
Since $u\in W^{2,2}_{0,rad}(B,w)$, $u\in W^{1,2}_{0,rad}(B)$ and by (\ref{eq:2.1}) we obtain
\begin{equation}\label{eq:2.3}
\|u_{n}-u\|_{W^{1,2}_{0,rad}(B)}\rightarrow0~~\mbox{as}~~n\rightarrow+\infty,
\end{equation}
and from (\ref{eq:2.2}),  (\ref{eq:2.3}),we have
 $$\|u-v\|_{W^{1,2}_{0,rad}(B)}\leq\|u_{n}-u\|_{W^{1,2}_{0,rad}(B)}+\|u_{n}-v\|_{W^{1,2}_{0,rad}(B)}\rightarrow 0~~\mbox{as}~~n\rightarrow+\infty,$$
 so $u=v~~\mbox{a .e in }~~B$ , $u\in \mathbf{W}$ and satisfies
 $$\|u_{n}-u\|_{W_{0,rad}^{2,2}(B,w)}\rightarrow 0~~\mbox{as}~~n\rightarrow+\infty.$$
Now we prove that $\|.\|$ is equivalent to $\|.\|_{W_{0,rad}^{2,2}(B,w)}$ in $\mathbf{W}$.
 $$
\| u \|^{2}_{W_{0}^{2,2}}= \| u \|^{2}_2 + \| \nabla u \|^{2}_2 + \displaystyle\int_{B}|\triangle u|^{2}w(x)dx.
 $$
For all $u \in W_{0,rad}^{2,2}(B)$, we have
 $$
\| u \|^{2}  = \displaystyle\int_{B}|\triangle u|^{2}w(x)dx\leq \| u \|_2 + \| \nabla u \|^{2}_2 +\displaystyle\int_{B}|\triangle u|^{2}w(x)dx
 $$
On the other hand, for all $u \in W_{0,rad}^{2,2}(B,w)$, by
 Poincar\'{e} inequality,
 $$
\|u\|^{2}_2 \leq C \|\nabla u\|^{2}_2,
 $$
and using the Green formula we get
$$
 \displaystyle \int_B \nabla u \nabla u  = - \displaystyle \int_B u \Delta
 u + \underbrace{\displaystyle\int_{\del B} u \frac{\partial u}{\partial n}}_{= 0} ~
 \leq  \displaystyle \Big| \int_B u \Delta u \Big|,
 $$
By  Young inequality, we get for all $\varepsilon>0$
 $$
 \displaystyle \Big| \int_B u \Delta u \Big| \leq
  \displaystyle \frac{1}{2 \varepsilon} \displaystyle \int_B |\Delta u|^2 + \displaystyle \frac{\varepsilon}{2} \displaystyle
 \int_B |u|^2,
 $$
 Again, by the  Poincarr\'{e} inequality and using the fact that $w(x)\geq 1,~~\mbox{for all}~~x\in B,$ we get
 $$
  \displaystyle \int_B \nabla u \nabla udx \leq   \displaystyle \frac{1}{2\varepsilon} \displaystyle\int_B |\Delta u |^2 dx
  + \displaystyle \frac{\varepsilon}{2}
  C^2 \displaystyle\int_B |\nabla u |^2dx \leq   \displaystyle \frac{1}{2\varepsilon} \displaystyle\int_B |\Delta u |^2 w(x)dx
  + \displaystyle \frac{\varepsilon}{2}
  C^2 \displaystyle\int_B |\nabla u |^2 dx.
 $$
Hence
 $$
 (1 - \displaystyle \frac{\varepsilon}{2} C^2) \displaystyle \int_B |\nabla u |^2dx ~ \leq ~
 \displaystyle \frac{1}{2\varepsilon} \displaystyle\int_ B |\Delta u |^2 w(x)dx,
  $$
wich implies that
\begin{equation}\label{nabla-Delta}
 \|\nabla u\|^{2}_2 \leq C \displaystyle\int_ B |\Delta u |^2 w(x)dx
\end{equation}
and it is easy to conclude.
\end{proof}
\subsection{The pass mountain geometry}
Since the nonlinearity $f(x,t)$ is critical at $+\infty$, there exist $ a,~ C>0$ positive constants and there exists $t_{1} >1$ such for that
\begin{equation}\label {eq:2.5}
|f(x,t)|\leq C e^{a ~t^{\gamma}}, ~~~~~~\forall |t| >t_{1}.
\end{equation}
So, the functional $\mathcal{J}$ given by (\ref{eq:1.6}) is well defined and of class $C^{1}$ .
\\
 In order to prove the existence of nontrivial solution to the problem (\ref{eq:1.1.1}), we will prove the existence of nonzero critical point of the functional $\mathcal{J}$ by using the following theorem which is introduced by Ambrosetti and Rabionowitz in \cite{AR} (Mountain Pass Theorem).
 \begin{definition}\label{def4.1}~~Let $(u_{n})$ be a sequence in a Banach space $E$ and $J\in C^{1}(E,\R)$ and let $c\in \R$. We say that the sequence $(u_{n})$ is a Palais-Smale sequence at level $c$ ( or  $(PS)_{c}$ sequence ) for the functional $J$ if
  $$J(u_{n})\rightarrow c~~\mbox{in}~~\R, ~~\mbox{as}~~ n\rightarrow+\infty$$
  and
  $$J'(u_{n})\rightarrow 0~~\mbox{in}~~E', ~~\mbox{as}~~ n\rightarrow+\infty.$$
We say that the functional $J$ satisfies the
Palais-Smale condition $(PS)_{c}$ at the level $c$ if
 every $(PS)_{c}$ sequence $(u_{n})$  is relatively compact in $E$.
  \end{definition}
%%%%%%%%%%%%%%%%%%%%%%%%%%%%%%%%%%%%%%%%%%%%%%%%%%%%%%%%%%%%%%%%%%%%%%%%%%%%%%%%
%%%%%%%%%%%%%%%%%%%%%%%%%%%%%%%%%%%%%%%%%%%%%%%%%%%%%%%%%%%%%%%%%%%%
\begin{theorem}\label{th2.11} \cite{AR}
Let $E$ be a Banach space and  $J: E \rightarrow \mathbb{R}$  a $C^{1}$ functional satisfying $J(0)=0$. Suppose that
  \begin{enumerate}
    \item [(i)] There exist $ \rho ,~~ \beta >0$ such that $\forall u \in \partial B(0,\rho), J(u) \geq \beta$;
    \item [(ii)] There exists $ x_{1} \in E$ such that $ \| x_{1} \|>\rho ~~\text{and }~~J(x_{1})<0$;
 \item [(iii)] $J$ satisfies the Palais-Smale condition $(PS)$,  that is for all sequence $(u_{n}$ in $E$ satisfying
 \begin{equation}\label{eq:2.6}
 J(u_{n}) \rightarrow d~~~~\mbox{as}~~n\rightarrow +\infty
  \end{equation}
  for some $d\in \R$ and
 \begin{equation}\label{eq:2.7}
\|J'(u_{n})\|_{\ast}~\rightarrow 0 ~~~~\mbox{as}~~n\rightarrow +\infty,
\end{equation}
the sequence  $(u_{n})$ is relatively compact.
 \end{enumerate}
    Then, $J$ has a critical point $u$ and the critical value  $c=J(u)$ verifies
    $$c:=\displaystyle{\inf_{\gamma\in \Gamma}}\displaystyle{\max_{t\in [0,1] }} J(\gamma (t))$$
    where $\Gamma :=\{\gamma\in C([0,1],X)~~\mbox{such that}~~\gamma(0)=0 ~~\mbox{ and }~~\gamma(1)=x_{1}\}$
and $c \geq \beta$.\\
  \end{theorem}
Before starting the proof of the geometric properties for the function $\mathcal{J}$, we recall the following  radial Lemma  introduced  in \cite{WZ}.
\begin{lemma}\cite{WZ}\label{lem2.1} Let $u$ be a radially symmetric
 function in $C_{0}^{2}(B)$. Then, we have
 $$|u(x)|\leq \displaystyle\frac{1}{2\sqrt{2}\pi}\bigg(\frac{~~|\log(\frac{e}{|x|})|^{1-\beta}-1}{1-\beta}\bigg)^{\frac{1}{2}}\|u\|\cdot$$

\end{lemma}

Since $ w(x)\geq 1,~~\mbox{for all}~~x\in B$, then the following embedding $W_{0}^{2,2}(B,w)\hookrightarrow W_{0}^{2,2}(B)\hookrightarrow L^{q}(B)$ are continuous and also compact for all $q \geq 2$. So there exists a constant $C>0$ such that $\|u\|_{2q}\leq c \|u\|$, for all $u\in \mathbf{W}$.\\
In the next Lemma, we prove that the $\mathcal{J}$ satisfies the first geometric property.
%%%%%%%%%%%%%%%%%%%%%%%%%%%%%%%%%%%%%%%%%%%%%%%%%%%%%%%%%%%%%%%%%%%%%%%%%%%%%%%%%%%%%%%%%%%%%%%%%%%%%%%%%%%%%%%%%
\begin{lemma}\label{lem2.2}
Suppose that $(H_{1})~~\mbox{and}~~(H_{4})$ hold. Then, there exist  $\rho,~\beta_{0}>0$  such that $\mathcal{J}(u)\geq \beta_{0}$ for all
 $u\in \mathbf{W}$ with $\|u\|=\rho$.
 \end{lemma}

 \begin{proof}
 It follows from the hypothesis $(H_{4})$ that  there exists $t_{0}>0$ and there exists  $\varepsilon \in (0,1)$  such that
\begin{equation}\label {eq:2.8}
F(x,t)\leq \frac{1}{2}\lambda_{1}(1-\varepsilon_{0})|t|^{2},~~~~~~\mbox{for}~~|t|< t_{2}.
\end{equation}
Indeed,  $$\displaystyle\limsup_{t\rightarrow 0}\frac{2 F(x,t)}{t^{2}}< \lambda_{1}$$
or $$\displaystyle\inf_{\tau>0}\sup\{\frac{2F(x,t)}{t^{2}},~~0<t<\tau\}< \lambda_{1}$$
Since this inequality is strict, then there exists $\varepsilon_{0}>0$ such that
$$\displaystyle\inf_{\tau>0}\sup\{\frac{2F(x,t)}{t^{2}},~~0<t<\tau\}< \lambda_{1}-\varepsilon_{0},$$
hence, there exists $t_{2}>0$ such that
$$\sup\{\frac{2F(x,t)}{t^{N}},~~0<t<t_{2}\}< \lambda_{1}-\varepsilon_{0}.$$
Hence$$\forall |t|<t_{2}~~~F(x,t)\leq \frac{1}{2}\lambda_{1}(1-\varepsilon_{0})t^{2}.$$

From $(H_{3})$ and (\ref{eq:2.5}) and for all $q>2$, there exist a  constant $C>0$ such that
\begin{equation}\label{eq:2.9}
F(x,t)\leq  C |t|^{q} e^{a~t^{\gamma}},~~\forall~|t| > t_{1}.
\end{equation}
So
\begin{equation}\label {eq:2.10}
F(x,t)\leq \frac{1}{2}\lambda_{1}(1-\varepsilon_{0})|t|^{2}+C |t|^{q}e^{a~t^{\gamma}},~~~~~~\mbox{for all}~~t\in \R.
\end{equation}
Since
$$\mathcal{J}(u)=\frac{1}{2}\|u\|^{2}-\int_{B}F(x,u)dx,$$
we get
$$\mathcal{J}(u)\geq \frac{1}{2}\|u\|^{2}- \frac{1}{2}\lambda_{1}(1-\varepsilon_{0})\|u\|^{2}-C\int_{B} |u|^{q}e^{a~u^{\gamma}}~dx.$$
But $\lambda_{1}\|u\|_{2}^{2}\leq \|u\|^{2}$ and from the  H\"{o}lder inequality, we obtain
\begin{equation}\label {eq:2.11}
\mathcal{J}(u)\geq \frac{ \varepsilon_{0}}{2}\|u\|^{2}-C  (\int_{B}e^{2a~|u|^{\gamma}}dx\Big)^{\frac{1}{2}}\|u\|^{q}_{2q}\cdot
\end{equation}
From the Theorem \ref{th1.1}, if we choose $u\in \mathbf{W}$ such that
\begin{equation}\label {eq:2.12}
2a \|u\|^{\gamma}\leq \alpha_{\beta},
\end{equation}
we get
$$\int_{B}e^{2a|u|^{\gamma}}dx=\int_{B}e^{2a\|u\|^{\gamma}(\frac{|u|}{\|u\|})^{\gamma}}dx<+\infty.$$
On the other hand $\|u\|_{2q}\leq C \|u\|$,  so
$$\mathcal{J}(u)\geq \frac{\varepsilon_{0}}{2}\|u\|^{2}-C\|u\|^{q},$$
for all $u\in \mathbf{W}$ satisfying (\ref{eq:2.12}).
Since $2<q$, we can choose $\rho=\|u\|>0$ as the maximum point of the function $g(\sigma)=\frac{\varepsilon_{0}}{2} \sigma^{2}-C\sigma^{q}$ on the interval $[0,(\frac{\alpha_{\beta}}{2a})^{\frac{1}{\gamma}}]$ and $\beta_{0}=g(\rho)$ , $\mathcal{J}(u) \geq\beta_{0}>0$.
\end{proof}

%%%%%%%%%%%%%%%%%%%%%%%%%%%%%%%%%%%%%%%%%%%%%%%%%%%%%%%%%%%%%%%%%%%%%%%%%%%%%%%%%%%%%%%%%%%%%%%%%%%%%%%%%%%%%%%%%%%%%%%%%%%%%%%
By the following Lemma, we prove the second geometric property for the functional $\mathcal{J}$.
\begin{lemma}\label{lem2.4}
Suppose that $(H_{1})$ and $(H_{2})$ hold. Let $\varphi_{1}$ be a normalized eigenfunction associated to $\lambda_{1}$ in $\mathbf{W}$. Then, $\mathcal{J}(t\varphi_{1})\rightarrow -\infty,~~\mbox{as}~~t\rightarrow+\infty$.
 \end{lemma}
\begin{proof}
It follows from the condition $(H_{2})$ that
$$f(x,t)=\frac{\partial}{\partial t}F(x,t)\geq \frac{1}{M}F(x,t),$$
for all $t\geq t_{0}$. So
\begin{equation*}\label {eq:2.13}
F(x,t)\geq C~ e^{\frac{t}{M}},~~\forall ~~t\geq t_{0}.
\end{equation*}
It follows that, there exist $b>\lambda_{1}$ and $C>0$ such that $F(x,t)\geq \frac{b}{2} t^{2} + C$ for all $t> 0$.
$$\mathcal{J}(t\varphi_{1})\leq \frac{t^{2}}{2}\|\varphi_{1}\|^{2}- \frac{b}{2}t^{2}\|\varphi_{1}\|_{2}^{2}-C|B|,$$
where $|B|=meas (B)=Vol(B)$.
Then, from the definition of $\lambda_{1}$, we get
$$\mathcal{J}(t\varphi_{1})\leq t^{2}\frac{\lambda_{1}-b}{2}\|\varphi_{1}\|_{2}^{2}<0~~\forall t>0.$$
So, the Lemma \ref{lem2.4} follows.
\end{proof}
%4%%%%%%%%%%%%%%%%%%%%%%%%%%%%%%%%%%%%%%%%%%%%%%%%%%%%%%%%%%%%%%%%%%%%%%%%%%%%%%%%%%%%%%%%%%%%%%%%%%%%%%%%%%%%%%%%%%%%%44444
\section{ The compactness analysis }

%%%%%%%%%%%%%%%%%%%%%%%%%%%%%%%%%%%%%%%%%%%%%%%%%%%%%%%%%%%%%%%%%%%%%%%%%%%%%%%%%%%%%%%%%%%%%%%%%%%%%%%%%%%
\subsection{Concentration Compactness Theorem}
In order to prove that the functional $\mathcal{J}$ satisfies the $(PS)$ condition, we need a lions type result \cite{L} about an improved Adam's inequality.
%%%%%%%%%%%%%%%%%%%%%%%%%%%%%%%%%%%%%%%%%%%%%%%%%%%%%%%%%%%%%%%%%%%%%%%%%%%%%%%%%%%%%%%%%%%%%%%%%%%%%%%%%%%%%%%%%%%ù
\begin{theorem}\label{lem3.1}  Let
$(u_{k})_{k}$  be a sequence in $\mathbf{W}$. Suppose that, \newline $\|u_{k}\|=1$, $u_{k}\rightharpoonup u$ weakly in $\mathbf{W}$, $u_{k}(x)\rightarrow u(x) ~~a.e~x\in B$,  and $u\not\equiv 0$. Then
$$\displaystyle\sup_{k}\int_{B}e^{p~\alpha_{\beta}
|u_{k}|^{\gamma}}dx< +\infty,~~\mbox{where}~~ \alpha_{\beta}=4[8\pi^{2}(1-\beta)]^{\frac{1}{1-\beta}},$$
for all $1<p<U(u)$ where $U(u)$ is given by:
 $$U(u):=\displaystyle \left\{
      \begin{array}{rcll}
&\displaystyle\frac{1}{(1-\|u\|^{2})^{\frac{\gamma}{2}}}& \mbox{ if }\|u\| <1\\
       &+\infty& \mbox{ if } \|u\|=1\\
 \end{array}
    \right.$$
\end{theorem}
%%%%%%%%%%%%%%%%%%%%%%%%%%%%%%%%%%%%%%%%%%%%%%%%%%%%%%%%%%%%%%%%%%%%%%
\begin{proof}
 Since $\left\|  u  \right\| \leq   \underset{ k}{ \liminf}\left\|   u_{k}  \right\|=1$, we will split the evidence into two cases.\\

\textbf{Case $1$} $:\left\|  u  \right\|<1$. We assume by contradiction for some $p_{1}<U(u)$, we have
$$
\sup_{k} \int_{B} \exp \left(\alpha_{\beta} p_{1} u_{k}^{\gamma}\right) d x=+\infty .
$$
Set
$$
B_{\mathcal{L}}^{k}=\left\{x\in B: u_{k}(x) \geq \mathcal{L}\right\}
$$
where $\mathcal{L}$ is a constant that we will choose later. Let $v_{k}=u_{k}-\mathcal{L}$.  we have
\begin{equation}\label {eq:3.1}
(1+a)^{q}\leq (1+\varepsilon) a^{q}+(1-\frac{1}{(1+\varepsilon)^{\frac{1}{q-1}}})^{1-q},~~\forall a\geq0,~~\forall\varepsilon>0~~\forall q>1.
\end{equation}
So, using (\ref{eq:3.2}), we get
\begin{equation}\label {eq:3.2}
      \begin{array}{rcll}
|u_{k}|^{\gamma}&=& |u_{k}-\mathcal{L}+\mathcal{L}|^{\gamma}\\
&\leq& (|u_{k}-\mathcal{L}|+|\mathcal{L}|)^{\gamma}\\
&\leq& (1+\varepsilon)|u_{k}-\mathcal{L}|^{\gamma}+\big(1-\frac{1}{(1+\varepsilon)^{\frac{1}{\gamma-1}}}\big)^{1-\gamma}|\mathcal{L}|^{\gamma}\\
&\leq&(1+\varepsilon) v_{k}^{\gamma}+C(\varepsilon, \gamma) \mathcal{L}^{\gamma} \cdot
 \end{array} \end{equation}
 We have

\begin{align*}
\int_{B} \exp \left(\alpha_{\beta} p_{1} u_{k}^{\gamma}\right) d x&=\int_{B_{L}^{k}} \exp \left(\alpha_{\beta} p_{1} u_{k}^{\gamma}\right) d x+\int_{B \backslash B_{\mathcal{L}}^{k}} \exp \left(\alpha_{\beta} p_{1} u_{k}^{\gamma}\right) d x\\
& \leq \int_{B_{\mathcal{L}}^{k}} \exp \left(\alpha_{\beta} p_{1} u_{k}^{\gamma}\right) d x
+c \exp \left(\alpha_{\beta} p_{1} \mathcal{L}^{\gamma}\right) \\
& \leq \int_{B_{\mathcal{L}}^{k}} \exp \left(\alpha_{\beta} p_{1} u_{k}^{\gamma}\right) d x+c(\mathcal{L}, \gamma ,|B|),
\end{align*}
and then
$$
\sup _{k} \int_{B_{\mathcal{L}}^{k}} \exp \left(\alpha_{\beta} p_{1} u_{k}^{\gamma}\right) d x=\infty .
$$
By (\ref{eq:3.2}) we have

\begin{align*}
\int_{B_{\mathcal{L}}^{k}} \exp \left(\alpha_{\beta} p_{1} u_{k}^{\gamma}\right) d x\leq & \exp \left(\alpha_{\beta} p_{1} C(\varepsilon, \gamma) \mathcal{L}^{\gamma}\right)
  \int_{B_{\mathcal{L}}^{k}} \exp \left((1+\varepsilon) \alpha_{\beta} p_{1} v_{k}^{\gamma}\right) d x .
\end{align*}
Since, $p_{1}<U(u)$, there exists $\varepsilon>0$ such that $\tilde{p}_{1}=(1+\varepsilon) p_{1}<U(u)$.
Thus
\begin{equation}\label{eq:3.3}
\sup _{k} \int_{B_{\mathcal{L}}^{k}} \exp \left(\tilde{p}_{1} \alpha_{\beta} v_{k}^{\gamma}\right) d x=\infty
\end{equation}

Now, we define
$$
T^{\mathcal{L}}(u)=\min \{\mathcal{L}, u\} \mbox{ and } T_{\mathcal{L}}(u)=u-T^{\mathcal{L}}(u)
$$
and choose $\mathcal{L}$ such that
\begin{equation}\label{eq:3.4}
\frac{1-\left\|  u  \right\|^{2}}{1-\left\|  T^{\mathcal{L}} u  \right\|^{2}}>\left(\frac{\tilde{p}_{1}}{U(u)}\right)^{\frac{2}{\gamma}} .
\end{equation}
We claim that
$$
\lim \sup_{k} \int_{B_{\mathcal{L}}^{k}}\omega(x)\left | \triangle v_{k}\right|^{2} d x<\left(\frac{1}{\tilde{p}_{1}}\right)^{\frac{2}{\gamma}} .
$$
If this is not the case, then up to a subsequence, we get
$$
\int_{B_{\mathcal{L}}^{k}}\omega(x)\left| \triangle v_{k}\right|^{2} d x=\int_{B}\omega(x)\left|\triangle  T_{\mathcal{L}} u_{k}\right|^{2} d x\geq\left(\frac{1}{\tilde{p}_{1}}\right)^{\frac{2}{\gamma}}+o_{k}(1) .
$$
Thus,

\begin{align*}
\left(\frac{1}{\tilde{p}_{1}}\right)^{\frac{2}{\gamma }}+\int_{B}\omega(x)\left| \triangle T^{\mathcal{L}} u_{k}\right|^{2} d x+o_{k}(1) & \leq \int_{B}\omega(x)\left|\triangle  T_{\mathcal{L}} u_{k}\right|^{2} d x+\int_{B \backslash B_{\mathcal{L}}^{k}}\omega(x)\left|  \triangle u_{k}\right|^{2} d x\\
&=\int_{B_{\mathcal{L}}^{k}}\omega(x)\left| \triangle u_{k}\right|^{2} d x+\int_{B \backslash B_{\mathcal{L}}^{k}}\omega(x)\left|\triangle  u_{k}\right|^{2} d x=1.
\end{align*}
For $\mathcal{L}>0$ fixed, $T^{\mathcal{L}} u_{k}$ is also bounded in $\mathbf{X}$. Therefore, up to a subsequence, $T^{\mathcal{L}} u_{k} \rightharpoonup$ $T^{\mathcal{L}} u$ weakly in $\mathbf{X}$ and $T^{\mathcal{L}} u_{k} \rightarrow T^{\mathcal{L}} u$ almost everywhere in $B$. By the lower semicontinuity of the norm in $\mathbf{X}$ and the last inequality, we have

$$  \tilde{p}_{1} \geq \frac{1}{\left( 1-  \underset{  k \rightarrow + \infty}{\lim \inf}
 \left\|  T^{\mathcal{L}} u_k  \right\|^{2}\right)^{\frac{\gamma}{2} }}  \geq \frac{1}{\left( 1-
 \left\|  T^{\mathcal{L}} u  \right\|^{2}\right)^{\frac{\gamma}{2} }} , $$

combining with (\ref{eq:3.4}), we obtain

$$  \tilde{p}_{1}  \geq \frac{1}{\left( 1-
 \left\|  T^{\mathcal{L}} u  \right\|^{2}\right)^{\frac{\gamma}{2} }} > \frac{\tilde{p}_{1} }{U(u)} \frac{1}{\left( 1-
 \left\|  T^{\mathcal{L}} u  \right\|^{2}\right)^{\frac{\gamma}{2} }} = \tilde{p}_{1}, $$

which is a contradiction. Therefore
$$
\limsup _{k} \int_{B_{\mathcal{L}}^{k}}\omega(x)\left|\triangle  v_{k}\right|^{2} d x<\left(\frac{1}{\tilde{p}_{1}}\right)^{\frac{2}{\gamma }} .
$$
By the Adam's inequality (\ref{eq:1.4}), we deduce that
$$
\sup _{k} \int_{B_{\mathcal{L}}^{k}} \exp \left(\tilde{p}_{1} \alpha_{\beta} v_{k}^{\gamma}\right) d x<\infty
$$
which is also a contradiction. The proof is finished in this case.\\

\textbf{Case $2$}  $:\left\|  u  \right\|=1$. We can then proceed as in case 1 and obtain
$$
\sup _{k} \int_{B_{\mathcal{L}}^{k}} \exp \left(\tilde{p}_{1}  \alpha_{\beta} v_{k}^{\gamma}\right)  d x =\infty
$$
where $\tilde{p}_{1}=(1+\varepsilon) p_{1}$. Then we have
$$
\lim \sup _{k} \int_{B_{L}^{k}}\omega(x)\left|\triangle  v_{k}\right|^{2} d x=\limsup _{k} \int_{B}\omega(x)\left|\triangle  T_{\mathcal{L}} u_{k}\right|^{2} d x\geq\left(\frac{1}{\tilde{p}_{1}}\right)^{\frac{2}{\gamma }}
$$
thus,

\begin{align*}
\left\|  T^{\mathcal{L}} u  \right\|^{2} \leq  \liminf _{k} \int_{B}\omega(x)\left|\triangle  T^{\mathcal{L}} u_{k}\right|^{2} d x\leq1
 -\limsup_{k} \int_{B}\left|\triangle  T_{\mathcal{L}} u_{k}\right|^{2} d x\leq 1-\left(\frac{1}{\tilde{p}_{1}}\right)^{\frac{2}{\gamma }} .
\end{align*}
On the other hand, since $\left\|  u  \right\|=1$, we can take $\mathcal{L}$ large enough such that
$$
\left\|  T^{\mathcal{L}} u  \right\|^{2}>1-\frac{1}{2}\left(\frac{1}{\tilde{p}_{1}}\right)^{\frac{2}{\gamma }}
$$
which is a contradiction, and the proof is complete in this case.

\end{proof}
\subsection{The Palais-Smale sequence}
%%%%%%%%%%%%%%%%%%%%%%%%%%%%%%%%%%%%%%%%%%%%%%%%%%%%%%%%%%%%%%%%%%%%%%%%%%%%%%%%%%%%%%%%%%%%%%%%%%%%%%%%%%%%%%%%%%%%%%%%
%%%%%%%%%%%%%%%%%%%%%%%%%%%%%%%%%%%%%%%%%%%%%%%%%%%%%%%%%%%%%%%%%%%%%%%%%%%%%%%%%%%%%%%%%%%%%%%%%%%%%%%%%%%%%%%%%%
The main difficulty in the approach to the critical problem of growth is the lack of compactness. Precisely,
the overall condition of Palais-Smale does not hold except for a certain level of energy. In the following
proposition, we identify the first level of compactness.
\begin{proposition}\label{Prop3.1} Suppose that $(H_{1}), (H_{2})$ and $(H_{3})$ hold. If the function $f(x,t)$ satisfies the condition (\ref{eq:1.5}) for some
$\alpha_{0} >0$, then the functional $\mathcal{J}$ satisfies the Palais-Smale condition $(PS)_{c}$ for any
$$c<\displaystyle\frac{1}{2}(\frac{\alpha_{\beta}}{\alpha_{0}})^{\frac{2}{\gamma}},$$
where $\alpha_{\beta}=4[8\pi^{2}(1-\beta)]^{\frac{1}{1-\beta}}.$

\end{proposition}
\begin{proof}

Consider a $(PS)_{c}$ sequence in $\mathbf{W}$, for some $c\in \R$, that is
\begin{equation}\label{eq:3.5}
\mathcal{J}(u_{n})=\frac{1}{2}\|u_{n}\|^{2}-\int_{B}F(x,u_{n})dx \rightarrow c ,~~n\rightarrow +\infty
\end{equation}
and
 \begin{equation}\label{eq:3.6}
|\langle \mathcal{J}'(u_{n}),\varphi \rangle|=\Big|\int_{B}w(x)\Delta u_{n}.\Delta \varphi dx -\int_{B}f(x,u_{n})\varphi  dx
\Big|\leq \varepsilon_{n}\|\varphi\|,
\end{equation}
for all $\varphi \in \mathbf{W}$, where $\varepsilon_{n}\rightarrow0$, when $n\rightarrow +\infty$.\\
Also, inspired by \cite{CRS}, it follows from  $(H_{2})$  that for all
$\varepsilon>0$ there exists $t_{\varepsilon}>0$ such that
\begin{equation}\label{eq:3.7}
 F(x,t)\leq \varepsilon t f(x,t),~~~~\mbox{for all}~~ |t|>t_{\varepsilon}~~\mbox{and uniformly in}~~ x\in B,
 \end{equation}
and so, by (\ref{eq:3.5}), for all $\varepsilon>0$ there exists a constant $C>0$
$$\displaystyle\frac{1}{2}\|u_{n}\|^{2}\leq C+\displaystyle\int_{B}F(x,u_{n})dx,$$
hence
$$
\displaystyle\frac{1}{2}\|u_{n}\|^{2}\leq  C +\displaystyle\int_{{|u_{n}|\leq t_{\varepsilon}}}F(x,u_{n})dx+ \varepsilon \int_{B}f(x,u_{n})u_{n}dx $$
and so, from (\ref{eq:3.6}), we get
$$\displaystyle\frac{1}{2}\|u_{n}\|^{2}\leq C_{1}+\varepsilon \varepsilon_{n}\|u_{n}\|+ \varepsilon \|u_{n}\|^{2},$$
for some constant $ C_{1}>0$.
Since
\begin{equation}\label{eq:3.8}
\displaystyle(\frac{1}{2}-\varepsilon)\|u_{n}\|^{2}\leq C_{1}+\varepsilon \varepsilon_{n}\|u_{n}\|,
 \end{equation}
 we deduce that the sequence $(u_{n})$ is bounded in $\mathbf{W}$. As consequence, there exists $u\in \mathbf{W}$ such that, up to subsequence,
 $u_{n}\rightharpoonup u $ weakly in $\mathbf{W}$, $u_{n}\rightarrow u$ strongly in $L^{q}(B)$, for all $1\leq q<4$ and $u_{n}(x)\rightarrow u(x)$ a.e. in $B$. \\
Furthermore, we have, from (\ref{eq:3.5}) and (\ref{eq:3.6}), that
\begin{equation}\label{eq:3.9}
0<\int_{B} f(x,u_{n})u_{n}\leq C,
 \end{equation}
and
 \begin{equation}\label{eq:3.10}
0<\int_{B} F(x,u_{n})\leq C.
 \end{equation}
Since by Lemma 2.1 in \cite {FMR}, we have
\begin{equation}\label{eq:3.11}
f(x,u_{n})\rightarrow f(x,u) ~~\mbox{in}~~L^{1}(B) ~~as~~ n\rightarrow +\infty,
 \end{equation}
then, it follows from $(H_{2})$ and the generalized Lebesgue dominated convergence theorem that
\begin{equation}\label{eq:3.12}
F(x,u_{n})\rightarrow F(x,u) ~~\mbox{in}~~L^{1}(B) ~~as~~ n\rightarrow +\infty.
 \end{equation}
So,
\begin{equation}\label{eq:3.13}
\displaystyle
\lim_{n\rightarrow+\infty}\|u_{n}\|^{2}=2(c+\int_{B}F(x,u)dx).
 \end{equation}
 Using (\ref{eq:3.5}), we have
\begin{equation}\label{eq:3.14}
\displaystyle\lim_{n\rightarrow+\infty}\int_{B}f(x,u_{n})u_{n}dx=2(c+\int_{B}F(x,u)dx).
 \end{equation}
 Then by $(H_{3})$ and (\ref{eq:3.6}), we get
\begin{equation}\label{eq:3.15}
\displaystyle
\lim_{n\rightarrow+\infty}2 \int_{B}F(x,u_{n})dx\leq \displaystyle
\lim_{n\rightarrow+\infty}\int_{B}f(x,u_{n})u_{n}dx= 2(c+\int_{B}F(x,u)dx).
\end{equation}
As a direct consequence from (\ref{eq:3.11}) and (\ref{eq:3.12}), we get  $c\geq0 $.\\
Also, it follows from (\ref{eq:3.5}), (\ref{eq:3.6}), (\ref{eq:3.11}) and (\ref{eq:3.12}), by passing to the limit, we obtain  that $u$ is a weak solution of the problem (\ref{eq:1.1.1}) that is
\begin{equation*}\label {eq:2.2}
\int_{B}\triangle u .\triangle\varphi~ w(x) dx =
\int_{B}f(x,u) \varphi dx,~~\forall~\varphi \in \mathbf{W}.
\end{equation*} Taking $\varphi=u$ as a test function, we get
$$\int_{B}|\triangle u|^{2}~ w(x) dx =
\int_{B}f(x,u)u dx\geq 2\int_{B}F(x,u)dx\cdot$$
~Hence $\mathcal{J}(u)\geq 0$ . We also have by the Fatou's lemma and (\ref{eq:3.12}) $$0\leq\mathcal{J}(u)\leq \frac{1}{2}\liminf_{n\rightarrow\infty} \|u_{n}\|^{2}-\int_{B}F(x,u)dx=c.$$

So, we will finish the proof by considering  three cases for the level $c$.\\\\
 %%%%%%%%%%%%%%%%%%%%%%%%%%%%%%%%%%%%%%%%%%%%%%%%%%%%%%%%%%%%%%%%%%%%%%%%%%%%%%%%%%%%%%%%%%%%%%%%%%%%%%%%%%%%%%%%%%%%%%%%%%%%%%%%%%%%%%%%ùùùùùù
{\it Case 1.} $c=0$.~~ In this case
 $$0\leq \mathcal{J}(u)\leq \liminf_{n\rightarrow+\infty}\mathcal{J}(u_{n})=0.$$
So, $$\mathcal{J}(u)=0$$
   and then by (\ref{eq:3.12})
$$\displaystyle\lim_{n\rightarrow +\infty}\frac{1}{2}\|u_{n}\|^{2}=\int_{B}F(x,u)dx=\frac{1}{2}\|u\|^{2}.$$
It follows that $u_{n}\rightarrow u~~\mbox{in}~~\mathbf{W}$.\\
%%%%%%%%%%%%%%%%%%%%%%%%%%%%%%%%%%%%%%%%%%%%%%%%%%%%%%%%%%%%%%%%%%%%%%%%%%%%%%%%%%%%%%%%%%%%%%%%%%%%%%%%%%%%%%%%%%%%%%%%%%%%%%%%%%%%%%%%%%%%%%%%%ùù
\noindent {\it Case 2.} $c>0$ and $u=0$. We prove that this case cannot happen.\\
From (\ref{eq:3.5}) and (\ref{eq:3.6}) with $v=u_{n}$, we have
$$\displaystyle\lim_{n\rightarrow +\infty}\|u_{n}\|^{2}=2c~~ \mbox{and}~~\displaystyle\lim_{n\rightarrow +\infty}\int_{B}f(x,u_{n})u_{n}dx=2c.$$
Again by (\ref{eq:3.6}) we have
$$\big| \|u_{n}\|^{2}-\int_{B}f(x,u_{n})u_{n}dx\big|\leq C\varepsilon_{n}.$$
First we claim that there exists $q>1$ such that
\begin{equation}\label{eq:3.16}
\int_{B}|f(x,u_{n})|^{q}dx\leq C,
\end{equation}
so
$$\|u_{n}\|^{2}\leq C\varepsilon_{n}+\big(\int_{B}|f(x,u_{n})|^{q}\big)^{\frac{1}{q}}dx(\int_{B}|u_{n}|^{q'}\big)^{\frac{1}{q'}}$$
where $q'$ the conjugate of $q$. Since $(u_{n})$ converge to $u=0$ in $L^{q'}(B)$,
$$\displaystyle\lim_{n\rightarrow +\infty}\|u_{n}\|^{2}=0$$
which in contradiction with $c>0$.\\\\
For the proof of the claim, since $f$ has subcritical or critical growth, for every
$\varepsilon>0$ and $q>1$ there exists $t_{\varepsilon}>0$ and $C>0$
such that for all $|t|\geq
t_{\varepsilon}$, we have
\begin{equation}\label{eq:3.17}
|f(x,t)|^{q}\leq
Ce^{\alpha_{0}(1+\varepsilon) t^{\gamma}}.
\end{equation}
Consequently,
$$\begin{array}{rcll}
\displaystyle\int_{B}|f(x,u_{n})|^{q}dx&=&\displaystyle\int_{\{|u_{n}|\leq
t_{\varepsilon}}|f(x,u_{n})|^{q}dx
+\int_{\{|u_{n}|>t_{\varepsilon}\}}|f(x,u_{n})|^{q}dx\\
 &\leq &\displaystyle2\pi^{2} \max_{B\times [-t_{\varepsilon},t_{\varepsilon}]}|f(x,t)|^{q}+ C\int_{B}e^{\alpha_{0}(1+\varepsilon)|u_{n}|^{\gamma}}\big)dx.
 \end{array}$$
Since $2c<\displaystyle(\frac{\alpha_{\beta}}{\alpha_{0}})^{\frac{2}{\gamma}}$, there exists
$\eta\in(0,\frac{1}{2})$ such that
$2c=\displaystyle(1-2\eta)\displaystyle(\frac{\alpha_{\beta}}{\alpha_{0}})^{\frac{2}{\gamma}}$.
On the other hand, $\|u_{n}\|^{\gamma}\rightarrow
(2c)^{\frac{\gamma}{2}}$, so there exists $n_{\eta}>0$ such that for
all $n\geq n_{\eta}$, we get $\|u_{n}\|^{\gamma}\leq
(1-\eta)\frac{\alpha_{\beta}}{\alpha_{0}}$.
Therefore,
$$\alpha_{0}(1+\varepsilon)(\frac{|u_{n}|}{\|u_{n}\|})^{\gamma}\|u_{n}\|^{\gamma}\leq
(1+\varepsilon)(1-\eta)\alpha_{\beta}\cdot$$
 We choose $\varepsilon >0$ small enough to get
 $$\alpha_{0}(1+\varepsilon) \|u_{n}\|^{\gamma}\leq \alpha_{\beta}\cdot $$
 Therefore, the second integral is
uniformly bounded in view of (\ref{eq:1.4}) and the claim is
proved.\\\\
%%%%%%%%%%%%%%%%%%%%%%%%%%%%%%%%%%%%%%%%%%%%%%%%%%%%%%%%%%%%%%%%%%%%%%%%%%%%%%%%%%%%%%%%%%%%%%%%%%%%%%%%%%%%%%%%%%%%%%%%%%%%%%%%%%%%%%%%%%%%%%%%%%%
{\it Case 3.} $c>0$ and $u\neq 0$.~~ In this case,
we claim that $\mathcal{J}(u)=c$ and therefore, we get
$$\lim_{n\rightarrow
+\infty}\|u_{n}\|^{2}=2\big(c+\int_{B}F(x,u)dx\big)=2\big(\mathcal{J}(u)+\int_{B}F(x,u)dx\big)=\|u\|^{2}.$$
Do not forgot that
 $$\mathcal{J}(u)\leq \frac{1}{2}\liminf_{n\rightarrow+\infty} \|u_{n}\|^{2}-\int_{B}F(x,u)dx=c.$$
Suppose that $\mathcal{J}(u)<c$. Then,
\begin{equation}\label{eq:3.18}
\|u\|^{\gamma}<(2\big(c+\int_{B}F(x,u)dx\big)\big)^{\frac{\gamma}{2}}.
\end{equation}
Set
$$v_{n}=\displaystyle\frac{u_{n}}{\|u_{n}\|}$$ and
 $$v=\displaystyle\frac{u}{(2\big(c+\displaystyle\int_{B}F(x,u)dx\big))^{\frac{1}{2}}}\cdot$$
We have $\|v_{n}\|=1$, $v_{n}\rightharpoonup v$  in $\mathbf{W}$, $v\not\equiv 0$ and $\|v\|<1$. So, by Theorem \ref{lem3.1}, we get
$$\displaystyle \sup_{n}\int_{B}e^{p \alpha_{\beta}|v_{n}|^{\gamma}}dx<\infty,$$ for
$1<p<U(v)=(1-\|v\|^{2})^{\frac{-\gamma}{2}}$.\\
 As in the case $(2)$, we are going to estimate $\displaystyle\int_{B}|f(x,u_{n})|^{q}dx$.\\
For  $\varepsilon>0$, one has
$$\begin{array}{rclll}
\displaystyle\int_{B}|f(x,u_{n})|^{q}dx&=&\displaystyle\int_{\{|u_{n}|\leq
t_{\varepsilon}}|f(x,u_{n})|^{q}dx
+\int_{\{|u_{n}|>t_{\varepsilon}\}}|f(x,u_{n})|^{q}dx\\
 &\leq &\displaystyle 2\pi^{2} \max_{B\times [-t_{\varepsilon},t_{\varepsilon}]}|f(x,t)|^{q}+ C\int_{B}e^{\alpha_{0}(1+\varepsilon )|u_{n}|^{\gamma}}dx\\
&\leq & C_{\varepsilon}+
 C\displaystyle\int_{B}e^{\alpha_{0}(1+\varepsilon)\|u_{n}\|^{\gamma}|v_{n}|^{\gamma}}\big)dx\leq C,
 \end{array}$$
provided that $\alpha_{0}(1+\varepsilon)\|u_{n}\|^{\gamma}\leq p~~
\alpha_{\beta}$ and
$1<p<U(v)=(1-\|v\|^{2})^{\frac{-\gamma}{2}}$.\\
Since
$$\displaystyle(1-\|v\|^{2})^{\frac{-\gamma}{2}}=\displaystyle\big(\frac{2(c+\int_{B}F(x,u)dx)}{2(c+\int_{B}F(x,u)dx)-\|u\|^{2})}\big)^{\frac{\gamma}{2}}=
\big(\frac{c+\int_{B}F(x,u)dx}{c-\mathcal{J}(u)}\big)^{\frac{\gamma}{2}}$$
and
$$\displaystyle\lim_{n\rightarrow+\infty}\|u_{n}\|^{\gamma}=(2\big(c+\int_{B}F(x,u)dx)\big)^{\frac{\gamma}{2}},$$
then, for large $n$,
 $$\alpha_{0}(1+\varepsilon)\|u_{n}\|^{\gamma}\leq \alpha_{0}(1+2\varepsilon)
(2\big(c+\int_{B}F(x,u)dx\big)^{\frac{\gamma}{2}}.$$
But $\mathcal{J}(u)\geq 0$ and $c<\displaystyle\frac{1}{2}(\frac{\alpha_{\beta}}{\alpha_{0}})^{\frac{2}{\gamma}}$ , then
if we choose $\varepsilon>0$ small enough such that $$ \frac{\alpha_{0}}{\alpha_{\beta}}(1+2 \varepsilon)
<\displaystyle\big(\frac{1}{2(c-\mathcal{J}(u))}\big)^{\frac{\gamma}{2}},$$
we get,
 $$
 (1+2\varepsilon)\displaystyle\big((c-\mathcal{J}(u)\big)^{\frac{\gamma}{2}}
<\frac{\alpha_{\beta}}{2^{\frac{\gamma}{2}}\alpha_{0}}\cdot
$$
 So, the sequence $(f(x,u_{n}))$ is bounded in $L^{q}$, $q>1$.\\
%%%%%%%%%%%%%%%%%%%%%%%%%%%%%%%%%%%%%%%%%%%%%%%%%%%%%%%%%%%%%%%%%%%%%%%%%%%%%%%%%%%%%%%%%ùù
 Since $\langle \mathcal{J}'(u_{n}),(u_{n}-u)\rangle=o(1)$,  we have from the boundedness of
$\{f(x,u_{n})\}$ in $L^{q}(B)$ for $q>1$, we can prove that
$u_{n}\rightarrow u$ strongly in $\mathbf{W}$. Indeed, we have
  $$ \|u_{n}-u\|^{2}=\langle u_{n},u_{n}-u\rangle-\langle u,u_{n}-u\rangle=\langle u_{n},u_{n}-u\rangle+o_{n}(1)\rightarrow 0~~as~~ n\rightarrow+\infty.$$
 From (\ref{eq:3.6}) and using the
H\"{o}lder inequality, we get
$$\begin{array}{rclllll}
\displaystyle|<u_{n},u_{n}-u>|&\leq& \displaystyle\varepsilon_{n}\|u_{n}-u\|+\big|\int_{B}f(x,u_{n})(u_{n}-u)dx\big|\\
 &\leq&C \displaystyle\varepsilon_{n}+\big(\int_{B}|f(x,u_{n}|^{q}dx\big)^{\frac{1}{q}}(\int_{B}|u_{n}-u|^{q'})^{\frac{1}{q'}}dx\rightarrow
0~~as~~ n\rightarrow+\infty .\end{array}$$
 Hence,
$$\displaystyle\lim_{n\rightarrow+\infty}\|u_{n}\|^{2}=2(c+\int_{B}F(x,u)dx)=\|u\|^{2}$$
and this contradicts (\ref{eq:3.18}). So, $\mathcal{J}(u)=c$ and  consequently, $u_{n} \rightarrow u$.

\end{proof}

%%%%%%%%%%%%%%%%%%%%%%%%%%%%%%%%%%%%%%%%%%%%%%%%%%%%%%%%%%%%%%%%%%%%%%%%%%%%%%%%%%%%%%%%%%%%%%%%%%%%%%%%%%%%
\section{Proof of Theorem \ref{th1.3}}
%%%%%%%%%%%%%%%%%%%%%%%%%%%%%%%%%%%%%%%%%%%%%%%%%%%%%%%%%%%%%%%%%%%%%%%%%%%%%%%%%%%%%%%%%%%%%%%%%%%%%%%%%%%%%%%%%

 In the sequel, we will estimate the minimax level of the energy $\mathcal{J}$.
We will  prove that the mountain pass level $c$ in Theorem \ref{th2.11} satisfies
$$c<\displaystyle\frac{1}{2}(\frac{\alpha_{\beta}}{\alpha_{0}})^{\frac{2}{\gamma}}\cdot$$
 For this purpose , we will prove that there exists $v_{0}\in \mathbf{W}$ such
  \begin{equation}\label{eq:4.1}
 \max_{t\geq0}\mathcal{J}(tv_{0})<\displaystyle\frac{1}{2}(\frac{\alpha_{\beta}}{\alpha_{0}})^{\frac{2}{\gamma}}\cdot
  \end{equation}
  \subsection{Adams functions}
Now, we will construct particular functions, namely the Adams functions.
  We consider the sequence defined for all $n\geq3$  by
  \begin{equation}\label{eq:4.2}
\end{equation}
$$w_{n}(x)=\displaystyle \left\{
      \begin{array}{rcllll}
&\displaystyle\bigg(\frac{\log (e \sqrt[4]{n})}{\alpha_{\beta}}\bigg)^{\frac{1}{\gamma}}-\frac{|x|^{2(1-\beta)}}{2\big(\frac{\alpha_{\beta}}{4n}\big)^{\frac{1}{\gamma}}\big(\log  (e\sqrt[4]{n})\big)^{\frac{\gamma-1}{\gamma}}}+\frac{1}{2(\frac{\alpha_{\beta}}{4})^{\frac{1}{\gamma}}\big(\log  (e\sqrt[4]{n})\big)^{\frac{\gamma-1}{\gamma}}}& \mbox{ if } 0\leq |x|\leq \frac{1}{\sqrt[4]{n}}\\\\\\
       &\displaystyle  \frac{\bigg(\log(\frac{e}{|x|})\bigg)^{1-\beta}}{\bigg(\frac{\alpha_{\beta}}{4}\log (e\sqrt[4]{n})\bigg)^{\frac{1}{\gamma}}} & \mbox{  if } \frac{1}{\sqrt[4]{n}}\leq|x|\leq \frac{1}{2}\\
       &\zeta_{n}&\mbox{  if }\frac{1}{2}\leq |x|\leq 1
 \end{array}
    \right.$$

     where $\zeta_{n}\in C^{\infty}_{0}(B)$ is such that\\
      $\zeta_{n}(\frac{1}{2})=\frac{1}{\big(\frac{\alpha_{\beta}}{16}\log (e^{4}n)\big)^{\frac{1}{\gamma}}}\big(\log 2e \big)^{1-\beta}$,
      $\displaystyle\frac{\partial \zeta_{n}}{\partial x}(\frac{1}{2})=\frac{-2(1-\beta)}{\bigg(\frac{\alpha_{\beta}}{4}\log (e\sqrt[4]{n})\bigg)^{\frac{1}{\gamma}}} \big(\log (2e)\big)^{-\beta}$ \newline  $\displaystyle\zeta_{n}=\frac{\partial \zeta_{n}}{\partial x}=0~~\mbox{on}~~\partial B$ and the functions $\xi_{n}$, $\nabla \xi_{n}$, $\Delta \xi_{n}$ are all $\displaystyle o\bigg(\frac{1}{\log (e\sqrt[4]{n})}\bigg)$.\\
     Let $v_{n}(x)=\displaystyle\frac{w_{n}}{\|w_{n}\|}$.
 We have, $v_{n}\in \mathbf{W}$ , $\|v_{n}\|^{2}=1.$ \\

 We compute $\Delta w_{n}(x)$, we get  \begin{equation*}\label{eq:4.2}\Delta w_{n}(x)=\displaystyle \left\{
      \begin{array}{rclll}
\displaystyle\frac{-(1-\beta)(4-2\beta)|x|^{-2\beta}}{\big(\frac{\alpha_{\beta}}{4n}\big)^{\frac{1}{\gamma}}\big(\log  (e\sqrt[4]{n})\big)^{\frac{\gamma-1}{\gamma}}}\mbox{      }\mbox{      }\mbox{      }\mbox{      }\mbox{      }\mbox{      }\mbox{      } & \mbox{ if } 0\leq |x|\leq \frac{1}{\sqrt[4]{n}}\\\\
       \displaystyle  \frac{-(1-\beta)\bigg(\log(\frac{e}{|x|})\bigg)^{-\beta}\bigg(2+\beta\big(\log \frac{e}{|x|}\big)^{-1}\bigg)}{\bigg(\frac{\alpha_{\beta}}{4}\log (e\sqrt[4]{n})\bigg)^{\frac{1}{\gamma}}} & \mbox{ if } \frac{1}{\sqrt[4]{n}}\leq|x|\leq \frac{1}{2}\\\\
       \triangle \zeta_{n}\mbox{      }\mbox{      }\mbox{      }\mbox{      }\mbox{      }
\mbox{      }\mbox{      }\mbox{      }\mbox{      }\mbox{      }\mbox{      }\mbox{      }
\mbox{      }\mbox{      }\mbox{      }\mbox{      }\mbox{      }\mbox{      }\mbox{      }
\mbox{      }\mbox{      }\mbox{      }\mbox{      }\mbox{      }\mbox{      }\mbox{      }
\mbox{      }\mbox{      }\mbox{      }\mbox{      } &\mbox{ if }\frac{1}{2}\leq |x|\leq 1
 \end{array}
    \right.
  \end{equation*}
  So, $$\|w_{n}\|^{2}=\underbrace{2\pi^{2}\int^{\frac{1}{\sqrt[4]{n}}}_{0}r^{3}|\Delta w_{n}(x)|^{2}\big(\log \frac{e}{r}\big)^{\beta}dr}_{I_{1}}+\underbrace{2\pi^{2}\int^{\frac{1}{2}}_{\frac{1}{\sqrt[4]{n}}}r^{3}|\Delta w_{n}(x)|^{2}\big(\log \frac{e}{r}\big)^{\beta}dr}_{I_{2}}+\underbrace{2\pi^{2}\int^{1}_{\frac{1}{2}}r^{3}|\Delta w_{n}(x)|^{2}\big(\log \frac{e}{r}\big)^{\beta}dr}_{I_{3}}$$
  we have,$$\begin{array}{rclll}\displaystyle I_{1}&= & \displaystyle 2\pi^{2}\frac{(1-\beta)^{2}(4-2\beta)^{2}}{\big(\frac{\alpha_{\beta}}{4n}\big)^{\frac{2}{\gamma}}\big(\log  (e\sqrt[4]{n})\big)^{\frac{2(\gamma-1)}{\gamma}}}\int^{\frac{1}{\sqrt[4]{n}}}_{0} r^{3-4\beta}\big(\log \frac{e}{r}\big)^{\beta}dr\\
&=&\displaystyle 2\pi^{2}\frac{(1-\beta)^{2}(4-2\beta)^{2}}{\big(\frac{\alpha_{\beta}}{4n}\big)^{\frac{2}{\gamma}}\big(\log  (e\sqrt[4]{n})\big)^{\frac{2(\gamma-1)}{\gamma}}}\left[  \frac{r^{4-4\beta}}{4-4\beta}(\log \frac{e}{r}\big)^{\beta} \right]^{\frac{1}{\sqrt[4]{n}}}_{0}\\&+& 2 \pi^{2}\frac{\beta(1-\beta)^{2}(4-2\beta)^{2}}{\big(\frac{\alpha_{\beta}}{4n}\big)^{\frac{2}{\gamma}}\big(\log  (e\sqrt[4]{n})\big)^{\frac{2(\gamma-1)}{\gamma}}}\displaystyle\int^{\frac{1}{\sqrt[4]{n}}}_{0}\frac{r^{4-4\beta}}{4-4\beta} \big(\log \frac{e}{r}\big)^{\beta-1} dr\\
&=&\displaystyle o\big(\frac{1}{\log e\sqrt[4]{n}}\big)\cdot\\

\end{array}$$
Also, $$\begin{array}{rclll}\displaystyle I_{2}&= & \displaystyle 2\pi^{2}\frac{(1-\beta)^{2}}{\big(\frac{\alpha_{\beta}}{4}\big)^{\frac{2}{\gamma}}\big(\log  (e\sqrt[4]{n})\big)^{\frac{2}{\gamma}}}\int_{\frac{1}{\sqrt[4]{n}}}^{\frac{1}{\frac{1}{2}} }\frac{1}{r}\big(\log \frac{e}{r}\big)^{-\beta}\big (2+\beta \big( \log \frac{e}{r}\big)^{-1}\big)^{2}dr\\
&=&\displaystyle -2\pi^{2}\frac{(1-\beta)^{2}}{\big(\frac{\alpha_{\beta}}{4}\big)^{\frac{2}{\gamma}}\big(\log  (e\sqrt[4]{n})\big)^{\frac{2}{\gamma}}}\left[ \frac{\beta^{2}}{-1-\beta}\big( \log \frac{e}{r}\big)^{-\beta-1}+4 \big( \log \frac{e}{r}\big)^{-\beta}+\frac{4}{1-\beta}\big( \log \frac{e}{r}\big)^{1-\beta}\right]^{\frac{1}{2}}_{\frac{1}{\sqrt[4]{n}}}\\
&=&\displaystyle 1+ o\big(\frac{1}{(\log e\sqrt[4]{n})^{\frac{2}{\gamma}}}\big)\cdot\\
\end{array}$$
and $I_{3}=\displaystyle  o\big(\frac{1}{(\log e\sqrt[4]{n})^{\frac{2}{\gamma}}}\big).$ Then $\|w_{n}\|^{2}=1+o\big(\frac{1}{(\log e\sqrt[4]{n})^{\frac{2}{\gamma}}}\big)$. Also,\\ for $0\leq |x|\leq \frac{1}{\sqrt[4]{n}}$, $\displaystyle v^{\gamma}_{n}(x)\geq \displaystyle\bigg(\frac{\log (e\sqrt[4]{n})}{\alpha_{\beta}}\bigg) +o(1)\cdot$
 \subsection{Min-Max level estimate}
We are going to the desired estimate.
%%%%%%%%%%%%%%%%%%%%%%%%%%%%%%%%%%%%%%%%%%%%%%%%%%%%%%%%%%%%%%%%%%%%%%%%%%%%%%%%%%%%%%%%%%%%%%%%%%%%%%

\begin{lemma}\label{lem4.2}~~For the sequence $(v_{n})$ identified by (\ref{eq:4.2}), there exists $n\geq1$ such that
  \begin{equation}\label{eq:4.3}
 \max_{t\geq 0}\mathcal{J}(tv_{n})<\displaystyle\frac{1}{2}(\frac{\alpha_{\beta}}{\alpha_{0}})^{\frac{2}{\gamma}}\cdot
       \end{equation}
\end{lemma}
\begin{proof} By contradiction, suppose that for all $n\geq1$,
$$ \max_{t\geq 0}\mathcal{J}(tv_{n})\geq\displaystyle\frac{1}{2}(\frac{\alpha_{\beta}}{\alpha_{0}})^{\frac{2}{\gamma}}\cdot$$
Therefore,
for any $n\geq1$, there exists $t_{n}>0$ such that
$$\max_{t\geq0}\mathcal{J}(tv_{n})=\mathcal{J}(t_{n}v_{n})\geq \displaystyle\frac{1}{2}(\frac{\alpha_{\beta}}{\alpha_{0}})^{\frac{2}{\gamma}}$$
and so,
$$\frac{1}{2}t_{n}^{2}-\int_{B}F(x,t_{n}v_{n})dx\geq \displaystyle\frac{1}{2}(\frac{\alpha_{\beta}}{\alpha_{0}})^{\frac{2}{\gamma}}\cdot$$
Then, by using $(H_{1})$
   \begin{equation}\label{eq:4.4}
t_{n}^{2}\geq\displaystyle(\frac{\alpha_{\beta}}{\alpha_{0}})^{\frac{2}{\gamma}}\cdot
       \end{equation}
On the other hand,
 $$\frac{d}{dt}\mathcal{J}(tv_{n})\big|_{t=t_{n}}=t_{n}-\int_{B}f(x,t_{n}v_{n})v_{n}dx=0,$$
that is
\begin{equation}\label{eq:4.5}
t_{n}^{2}=\int_{B}f(x,t_{n}v_{n})t_{n}v_{n}dx.
\end{equation}
Now, we claim that the sequence $(t_{n})$ is bounded in $(0,+\infty)$. \\Indeed,
 it follows from $(H_{5})$ that for all $\varepsilon>0$, there exists
$t_{\varepsilon}>0$ such that
\begin{equation}\label{eq:4.6}
f(x,t)t\geq
(\gamma_{0}-\varepsilon)e^{\alpha_{0}t^{\gamma}}~~\forall
|t|\geq t_{\varepsilon},~~\mbox{uniformly in}~~x\in B.
\end{equation}
Using (\ref{eq:4.4}) and (\ref{eq:4.5}), we get
\begin{equation*}\label{eq:3.13}
t^{2}_{n}=\int_{B}f(x,t_{n}v_{n})t_{n}v_{n}dx\geq
\int_{0\leq |x|\leq \frac{1}{\sqrt[4]{n}}}f(x,t_{n}v_{n})t_{n}v_{n}dx
\cdot\end{equation*}
Since $$\frac{t_{n}}{\|w_{n}\|}\big(\frac{\log e\sqrt[4]{n} }{\alpha_{\beta}}\big)^{\frac{1}{\gamma}} \rightarrow\infty~~\mbox{as}~~n\rightarrow+\infty,$$ then
it follows from  (\ref{eq:4.6}) that for all $\varepsilon >0$, there exists $n_{0}$ such that for all $n\geq n_{0}$
\begin{equation*}\label{eq:4.13}
t^{2}_{n}\geq \displaystyle (\gamma_{0}-\varepsilon)\int_{0\leq |x|\leq \frac{1}{\sqrt[4]{n}}}e^{\alpha_{0}t^{\gamma}_{n}v^{\gamma}_{n}}dx
  \end{equation*}
  \begin{equation}\label{eq:4.7}
 t^{2}_{n}\geq \displaystyle 2\pi^{2}  (\gamma_{0}-\varepsilon)\int_{0}^{\frac{1}{\sqrt[4]{n}}}r^{3}e^{\alpha_{0}t^{\gamma}_{n}( \displaystyle\big(\frac{\log (e\sqrt[4]{n})}{\alpha_{\beta}}\big) +o(1)}dr
.\end{equation}
Hence,
 \begin{equation*}\label{eq:4.13}
1 \geq  \displaystyle 2\pi^{2}  (\gamma_{0}-\varepsilon)~~\displaystyle e^{\alpha_{0}t^{\gamma}_{n}( \displaystyle\big(\frac{\log (e\sqrt[4]{n})}{\alpha_{\beta}}\big) +o(1))-3\log n-2\log t_{n}}.
 \end{equation*}
 Therefore $(t_{n})$ is bounded. Also, we have from the formula  (\ref{eq:4.5}) that
$$\displaystyle\lim_{n\rightarrow+\infty} t_{n}^{2}\geq\displaystyle(\frac{\alpha_{\beta}}{\alpha_{0}})^{\frac{2}{\gamma}}\cdot$$
Now, suppose that
$$\displaystyle\lim_{n\rightarrow+\infty} t_{n}^{2}>\displaystyle(\frac{\alpha_{\beta}}{\alpha_{0}})^{\frac{2}{\gamma}},$$
then for $n$ large enough, there exists some $\delta>0$ such that $ t_{n}^{\gamma}\geq \frac{\alpha_{\beta}}{\alpha_{0}}+\delta$. Consequently the right hand side of (\ref{eq:4.7}) tends to infinity and this contradicts the boudness of  $(t_{n})$. Since $(t_{n})$ is bounded, we get
\begin{equation}\label{eq:4.8}
 \displaystyle
 \lim_{n\rightarrow+\infty}t_{n}^{2}=\displaystyle(\frac{\alpha_{\beta}}{\alpha_{0}})^{\frac{2}{\gamma}}\cdot
 \end{equation}
 %%%%%%%%%%%%%%%%%%%%%%%%%%%%%%%%%%%%%%%%%%%%%%%
Let $$\mathcal{A}_{n}=\{x\in B| t_{n}v_{n}\geq
t_{\varepsilon}\}~~\mbox{and}~~\mathcal{C}_{n}=B\setminus \mathcal{A}_{n},$$
$$\begin{array}{rclll} t_{n}^{2}&=&\displaystyle\int_{B}f(x,t_{n}v_{n})t_{n}v_{n}dx=\int_{\mathcal{A}_{n}}f(x,t_{n}v_{n})t_{n}v_{n}dx+\int_{\mathcal{C}_{n}}f(x,t_{n}v_{n})t_{n}v_{n} $$\\
&\geq& \displaystyle(\gamma_{0}-\varepsilon)\int_{\mathcal{A}_{n}}e^{\alpha_{0}t_{n}^{\gamma}v_{n}^{\gamma}}dx + \int_{\mathcal{C}_{n}}f(x,t_{n}v_{n})t_{n}v_{n}dx\\
&=&\displaystyle(\gamma_{0}-\varepsilon)\int_{B}e^{\alpha_{0}t_{n}^{\gamma}v_{n}^{\gamma}}dx-
(\gamma_{0}-\varepsilon)\int_{\mathcal{C}_{n}}e^{\alpha_{0}t_{n}^{\gamma}v_{n}^{\gamma}}dx\\ &+&\displaystyle\int_{\mathcal{C}_{n}}f(x,t_{n}v_{n})t_{n}v_{n}dx.
\end{array}$$
Since $v_{n}\rightarrow 0 ~~\mbox{a.e in }~~B$,
$\chi_{\mathcal{C}_{n}}\rightarrow1~~\mbox{a.e in}~~B$, therefore using the dominated convergence
theorem, we get $$\displaystyle\int_{\mathcal{C}_{n}}f(x,t_{n}v_{n})t_{n}v_{n}dx\rightarrow 0~~\mbox{and}~~\int_{\mathcal{C}_{n}}e^{\alpha_{0}t_{n}^{\gamma}v_{n}^{\gamma}}dx\rightarrow \frac{\pi^{2}}{2}\cdot$$Then,$$\ \lim_{n\rightarrow+\infty} t_{n}^{2}=\displaystyle(\frac{\alpha_{\beta}}{\alpha_{0}})^{\frac{2}{\gamma}}\geq(\gamma_{0}-
\varepsilon)\lim_{n\rightarrow+\infty}\int_{B}e^{\alpha_{0}t_{n}^{\gamma}v_{n}^{\gamma}}dx-(\gamma_{0}-
\varepsilon)\frac{\pi^{2}}{2}\cdot$$
On the other hand,
$$\int_{B}e^{\alpha_{0}t_{n}^{\gamma}v_{n}^{\gamma}}dx \geq \int_{\frac{1}{\sqrt[4]{n}}\leq |x|\leq \frac{1}{2}}e^{\alpha_{0}t_{n}^{\gamma}v_{n}^{\gamma}}dx+\int_{\mathcal{C}_{n}} e^{\alpha_{0}t_{n}^{\gamma}v_{n}^{\gamma}}dx\cdot$$

Then, using (\ref{eq:4.4})  $$\lim_{n\rightarrow+\infty} t_{n}^{2}\geq \lim_{n\rightarrow+\infty}\displaystyle(\gamma_{0}-\varepsilon)\int_{B}e^{\alpha_{0}t_{n}^{\gamma}v_{n}^{\gamma}}dx\geq \displaystyle\lim_{n\rightarrow+\infty}(\gamma_{0}-\varepsilon) 2\pi^{2}\int^{\frac{1}{2}}_{\frac{1}{\sqrt[4]{n}}}r^{3} e^{\frac{4 \big(\log \frac{e}{r}\big)^{2}}{\log (e\sqrt[4]{n})\|w_{n}\|^{\gamma}}} dr$$
and, making the change of variable $$s=\frac{4\log \frac{e}{r}}{\log (e\sqrt[4]{n})\|w_{n}\|^{\gamma}},$$
we get
$$\begin{array}{rclll}\displaystyle \lim_{n\rightarrow+\infty} t_{n}^{2}&\geq & \displaystyle\lim_{n\rightarrow+\infty}\displaystyle(\gamma_{0}-\varepsilon)\int_{B}e^{\alpha_{0}t_{n}^{\gamma}v_{n}^{\gamma}}dx\\\\
&\geq&\displaystyle \lim_{n\rightarrow+\infty}2\pi^{2}(\gamma_{0}-\varepsilon)\frac{\|w_{n}\|^{\gamma}\log (e\sqrt[4]{n})}{4}e^{4}\int^{\frac{4}{\|w_{n}\|^{\gamma}}}_{\frac{4 \log 2e}{\|w_{n}\|^{\gamma}\log (e\sqrt[4]{n})}}\displaystyle e^{\frac{\|w_{n}\|^{\gamma}\log (e\sqrt[4]{n})}{4}~(s^{2}-4s)}ds\\\\
&\geq&\displaystyle \lim_{n\rightarrow+\infty}2\pi^{2}(\gamma_{0}-\varepsilon)\frac{\|w_{n}\|^{\gamma}\log (e\sqrt[4]{n})}{4}e^{4}\int^{\frac{4}{\|w_{n}\|^{\gamma}}}_{\frac{4\log 2e}{\|w_{n}\|^{\gamma}\log (e\sqrt[4]{n})}}e^{-\frac{\|w_{n}\|^{\gamma}\log (e\sqrt[4]{n})}{4}~4s}ds\\
&=&\displaystyle\lim_{n\rightarrow+\infty}(\gamma_{0}-\varepsilon)\frac{\pi^{2}}{2}e^{4}(-e^{-4 \log e\sqrt[4]{n}}+e^{-4\log (2e)}\big)\\
&=&\displaystyle(\gamma_{0}-\varepsilon)\frac{\pi^{2}e^{4(1-\log 2e)}}{2}=\displaystyle(\gamma_{0}-\varepsilon)\frac{\pi^{2} }{32}\cdot
\end{array}$$

It follows that
\begin{equation*}\label{eq:4.15}
\displaystyle(\frac{\alpha_{\beta}}{\alpha_{0}})^{\frac{2}{\gamma}}\geq (\gamma_{0}-\varepsilon)\frac{\pi^{2}}{32}
\end{equation*}
for all $\varepsilon>0$. So,
$$\gamma_{0}\leq \frac{1024(1-\beta)}{\alpha^{1-\beta}_{0}}, $$
   which is in contradiction with  the condition $(H_{5})$. \\Now by Proposition \ref{Prop3.1}, the functional $\mathcal{J}$ satisfies the $(PS)$ condition at a level $c<\displaystyle\frac{1}{2}(\frac{\alpha_{\beta}}{\alpha_{0}})^{\frac{2}{\gamma}}$ . So, by  Lemma \ref{lem2.2} and Lemma \ref{lem2.4}, we deduce that the functional $\mathcal{J}$ has a nonzero critical point $u$ in $\mathbf{W}$. From maximum principle, the solution $u$ of the problem  (\ref{eq:1.1.1}) is positive. The Theorem \ref{th1.3} is proved.
   \end{proof}

\end{document}